\documentclass[12pt, a4paper]{amsart}
\usepackage{amssymb,amsfonts,amsthm}
\usepackage[margin=1in]{geometry}
\usepackage{epsfig}
\usepackage{graphicx}
\numberwithin{equation}{section}
\usepackage{cite}
\graphicspath{{./graphics/}}
\input xy
\xyoption{all}  
\usepackage{url}
\usepackage{float}
\usepackage{setspace}
\usepackage{varioref}
\usepackage{hyperref}
\hypersetup{
    colorlinks=true,
    linkcolor=blue,
    filecolor=magenta,      
    urlcolor=cyan,
    citecolor=blue,
}
 
\usepackage{cleveref}
\crefformat{section}{\S#2#1#3}
\crefformat{subsection}{\S#2#1#3}
\usepackage{fancyhdr}
\usepackage[usenames,dvipsnames]{color}
\usepackage{refstyle}

\theoremstyle{plain}
\newtheorem{prop}{Proposition}[section]
\newtheorem{Theo}[prop]{Theorem}

\newtheorem{claim}[prop]{Claim}

\theoremstyle{definition}
\newtheorem{defi}[prop]{Definition}

\numberwithin{equation}{section}

\newcommand{\C}{\mathbb C}
\newcommand{\N}{\mathbb N}

\newcommand{\Z}{\mathbb Z}

\usepackage{cases}

\begin{document}
\author{Veronica Kalicki, Juan Morales, Ryan Ostrander}
\date{\today}
\title{Visualization of Abel's Impossibility Theorem}
\begin{abstract} In this paper we construct a visualization of the Abel's Impossibility Theorem also known as the Abel-Ruffini Theorem. Using the canvas object in JavaScript along with the p5.js library, and given any expression that uses analytic functions and radicals one can always construct closed paths such that the expression evaluated at the coefficients of a general polynomial returns to it's initial position, while the roots of the polynomial undergo a non-trivial permutation. Hence, such expression does not reconstruct the roots from the coefficients. Using the visualization we begin by considering the necessity of radicals to solve second degree polynomial equations and build towards degree five polynomial equations. In eventuality our program shows that there is no formula for an arbitrary fifth degree polynomial equation that uses analytic functions, finite field operations, and radicals that reconstructs the roots of the polynomial from it's coefficients. This theorem was partially completed by Paolo Ruffini in 1799 and completed by Niels Abel in 1824.
\end{abstract}

\maketitle
\tableofcontents
\newpage

\section{Introduction} 

 For a very long time, mathematicians attempted to find an expression for the roots of a fifth degree polynomial in terms of it’s coefficients and using radicals and finite field operations without success. That is until the discovery of the proof of the following theorem in 1824 by the Norwegian mathematician Niels Henrik Abel. The results in this paper are a computer representation this theorem. Using JavaScript, a computer simulation was made in order to demonstrate Abel's Theorem \cite{ostrander_morales_kalicki_2019}. 
Given the roots of an arbitrary monic polynomial we can find the corresponding coefficients using \textit{Vieta's} formulae for the cases of a second, third, fourth, and fifth degree polynomial equation. This was made to graphically express a mathematical result. The motivation behind the following project is based on the fact that computers can be used as mathematical aids and therefore we want to visualize and illustrate the theorem in question. 

In \cref{sec:idea} we discuss the main idea behind \textit{Arnold's} proof of \textit{Abel's} theorem, in \cref{sec:applyjava} we develop the general intuition and motivation for the program, in \cref{sec:quadratic} we illustrate the need for multi-valued functions in order to solve quadratic equations, in \cref{sec:commutator} we address the properties of radical functions, non-trivial permutations of the roots of a general polynomial, and the image of closed paths under radicals. In \cref{sec:cubic} we discuss the case of a third degree polynomial equation, in \cref{sec:symmetricgroup} we briefly list important properties of group theory, the symmetric and alternating groups, and their relevance. In \cref{sec:quartic} we visually illustrate the case for a four degree polynomial equation, and in \cref{sec:main} we state the main theorem. Lastly, in \cref{sec:impossibility} we discuss the main claim of this paper and in \cref{sec:javascript} we discuss the basic construction of the program the libraries used, and the overall implementation.

\begin{Theo}[Abel's Theorem]
The general algebraic equation with one unknown of degree
greater than 4 is insoluble in radicals, i.e. there does not exist a formula, which expresses the roots
of a general equation of degree greater than four in terms of the coefficients involving the operations
of addition, subtraction, multiplication, division, raising to a natural degree, and extraction of roots
of natural degree.\cite{alekseev2004abel}
\end{Theo}
Vladimir Igorevich Arnold of Moscow University based this proof on a series of lectures to undergraduate students developed through a series of problems given out to students starting with concepts from Gorup theory and building up to Monodromy groups.  
\section{Sketch of Arnold's Proof}\label{sec:idea}
We begin by introducing the fundamental idea behind \textit{Arnold's} proof, \textit{Riemann Surfaces} of the function $w=\sqrt[n]{z}$. The idea is to assign to each multi-valued function of a complex variable a Galois group so that it can be shown that the designated Galois group which expresses the roots of a given equation in terms of a parameter $z$ is not the Galois group for a function expressed in radicals. Thus in the spirit of illustrating the flexibility and importance of this proof by \textit{Arnold} we build the required intuition that leads to the proof of \textit{Abel's} theorem. 

\flushleft{\textbf{Riemann Surfaces for the functions $w=\sqrt[n]{z}$}.}

\
\\
We begin by considering multivalued functions and the construction of their Riemann surfaces. Consider the multivalued function $w=\sqrt{z}$. On the $z-plane$ we cut the negative part of the real axis starting at the origin to $-\infty$ such that for all $z\in \C -\{(-\infty,0)\}$ and all values of $w=\sqrt{z}$ which lie on the right-half plane defines a continuous single valued function on cut plane given by $w_1=\sqrt{z}$. Similarly we may choose values of $z$ that do not lie on the cut and the values of $w=\sqrt{z}$ that lie on the left half of the complex plane to obtain a continuous single valued function on the cut plane given by $w_2=\sqrt{z}$. Making two copies of the $z-plane$ and making the same cut as before we get what the author calls sheets where we define on each sheet the functions $w_1$ and $w_2$, respectively. If we then glue the cuts of each sheet in a unique manner we end up with what is called as the Riemann surface of the function $w=\sqrt{z}$. However, a more intuitive approach to understand this construction is to think about the function $w=\sqrt{z}$ not defined on the complex plane, but rather on an multi-level parking garage with a continuous ramp. Now suppose we move around the origin in the counterclockwise direction $n$ times. Then $z^t \mapsto z^{2\pi i nt} $, for $t\in \C$, and if $t$ is a rational number of the form $m/n$ where $m,n$ are co-primes then $e^{2\pi i nt}=e^{2\pi im}=1$. This means that as we wind about the origin $n$ times $z^t$ returns to it's original value, and in terms of our multi-level parking garage this means that as we enter the first level of the parking garage (the $1^{st}$ Riemann sheet) and as we wind around the origin we ``move up'' to the second, third,\ldots, $n$th level of the parking garage, but at the $(n+1)$th level we end up back at the first level.

\vspace{0.5em}

This process can be extended to any function of the form $w=\sqrt[n]{z}$ by making non-intersecting cuts from all branch points to infinity and gluing in a specific manner along the cuts, giving rise to single-valued continuous branches of the function $w=\sqrt[n]{z}$. More generally the types of functions that satisfy the above property follows from another property called the monodromy-property.                

We now want to assign a certain permutation group to each Riemann surface by letting $g_1,g_2,\ldots,g_s$ be the permutations of a certain Riemann surface corresponding to counterclockwise closed paths around all the branch points of the given function. Then the subgroup generated by the elements $\{g_1,\ldots, g_s\}$ is called the permutation group of the sheets for the given Riemann surface.
\
\\
Although, some of the arguments given by \textit{Arnold} require a general understanding of \textit{Galois} theory, we omit such explanations as they do not built any intuition of the topological and geometric processes that occur in our visualization. However, the final step of this proof sketch comes from the following: 
Consider the equation
\begin{equation}\label{eqn}
3w^5-25w^3+60w-z=0,
\end{equation} 
 where $z$ is a parameter such that for each complex value of $z$ we find all the complex roots of $w$ of the above equation \cite{alekseev2004abel}. Thus, if the values of $w(z)$ that express the roots of equation (\ref{eqn}) in terms of $z$ are parts of the values of the function $w_1(z)$, that can be expressed in radicals, then the Riemann surface for the function $w(z)$ is isolated from the Riemann surface of the function $w_1(z)$. Further, if $G$ is the Galois group of the function $w_1(z)$ then for every permutation of $G$ there corresponds a permutation of the five sheets of the Riemann surface of the function $w(z)$, which defines a group homomorphism $\phi:G\rightarrow S_5$. Since $S_5$ is not soluble, then $G$ is not soluble, but a function that is expressible by radicals has soluble Galois group, which leads to a contradiction. 
\section{Visualization}\label{sec:applyjava} 

Using JavaScript we developed an animated webpage application that allows users to visually understand the main argument of \textit{Arnold’s} proof. Specifically, it shows that given any expression $f : \{a_0, . . . , a_4\}\rightarrow \C^5$ that uses analytic functions and radicals, one can construct a closed path in the space $Poly_5(C)$ of Monic fifth degree polynomials, such that all values of $f$ return to their original positions, while the roots $z_1, . . . , z_5$ undergo a non-trivial permutation; therefore such $f$ cannot reconstruct the roots $z_1, . . . , z_5$ from the coefficients $a_0,...,a_4$. 

\begin{defi}
Recall that for any non-zero $z \in \C$ and $n\in \N$ there are precisely $n$ complex numbers $w$ with $w^n=z$.
\begin{equation}\label{polar}
z=r\cdot e^{i\theta},\qquad w=\sqrt[n]{r}\cdot e^{\frac{i}{n}(\theta+2k\pi)}\qquad k=0,1,\ldots ,n-1.
\end{equation}
Let $\gamma:[a,b]\rightarrow \C\setminus \{0\}$ be a closed path starting and ending at $z$. 
Then there are precisely $n$ paths $\omega_k:[a,b]\rightarrow \C\setminus \{0\}$ that trace the $n$th roots of $\gamma(t)$:
$$w_k(t)^n=\gamma(t)\qquad t\in [a,b],\quad k=0,1,\dots,n-1.$$
Note that while $\gamma$ is closed $\gamma(a)=z=\gamma(b)$ the paths $\omega_k$ need not be closed, yet the map
$$\omega_0(a), \dots,\omega_{n-1}(a)\ \mapsto\  \omega_0(b),\dots,\omega_{n-1}(b)$$
is always a cyclic permutation of the $n$ roots of the base point $z$.
\end{defi}

\subsection{Necessity of Radicals For Solving Quadratic Expressions}\label{sec:quadratic}
\
\\
We begin by considering the case for a \textit{Monic} polynomial $p(z) \in {\rm Poly}_2(\mathbb{C})$\footnote{Screen shot of the JavaScript simulation for the case of a quadratic polynomial (\figref{quad})}. The visualization specifically allows the user to understand the underlying ideas of why we need a multi-valued function, such as the radical function, to solve degree two polynomial equations. Explicitly, we induce closed paths on the roots of the polynomial equation, which in turn induces closed paths on the coefficients of the polynomial, which are symmetric expressions in terms of the roots. Hence, for any hypothetical expression that claims to express the roots of a general degree two polynomial equation in terms of the coefficients of the polynomial, analytic functions, and finite field operations, but no radicals, the visualization explicitly shows that this hypothetical expression does not trace the paths of the roots. 
\begin{figure}[htp]
	\centering
	\includegraphics[scale=0.5]{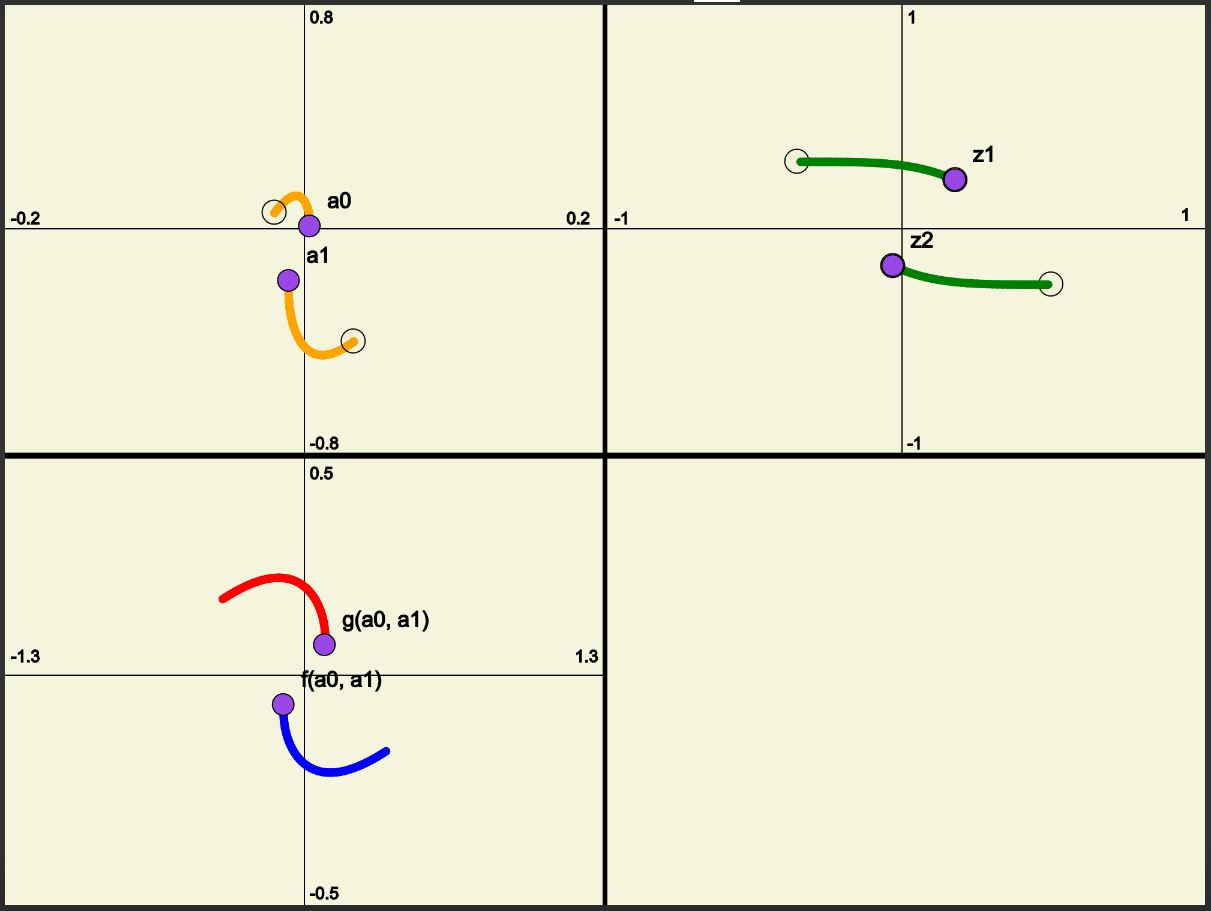}
	\caption[Quadratic]{Case for Monic $p(z)$ in ${\rm Poly}_2(\C)$}
	\label{fig:quad}
	\end{figure}

Formally we show that there is no formula for the roots $z_1, z_2$ of a general Monic polynomial $p(z)\in {\rm Poly}_2(\C) $ in terms of analytic functions $f,g:\{a_0,a_1\}\rightarrow \C$  
	such that $f(a_0,a_1)=z_1$ and $g(a_0,a_1)=z_2$ for a general quadratic equation.
	\begin{claim}
	There is no construction of a quadratic formula using only finite field operations, analytic functions, and the coefficients of the polynomial.
	\end{claim}
	\begin{proof}
	Suppose otherwise. Then, using \textit{Vieta's} formulae we find that the coefficients $a_0, a_1$ given by $a_0=z_1z_2$ and $a_1=-(z_1+z_2)$,
	which are symmetric expressions in $z_1,z_2$. 
	Starting from distinct points $z_1,z_2$ we can continuously move them until they change places $z_1\to z_2$, $z_2\to z_1$. 
	Under this motion, each of the coefficients $a_0,a_1$ follows a closed path and the functions $f(a_0,a_1)$ and $g(a_0,a_1)$ also follow closed paths. This, however, contradicts the assumption that $f$ and $g$ trace the roots $z_1,z_2$, which interchanged places. 
Therefore any formula requires the use of a multi-valued function mainly the \textit{quadratic formula}.  	
	\end{proof} 
\subsection{Radical, Functions, and a Commutator}\label{sec:commutator}
\
\\	
Before considering higher degree equations we must first address the fact that we are allowed, in general, to use multi-valued functions such as radicals in order to construct a formula for the roots of a polynomial equation. We first show that for certain types of closed paths and continuous functions, the paths that trace the values of $\sqrt[n]{\cdot}$ is a closed loop but at the same time they induce non-trivial permutations of the roots of any polynomial equation.    
\begin{claim}
Suppose that $\beta$ and $\gamma$ are two closed loops that start and end at the same point $a$. Then for any continuous function $f:\{a_0,\ldots a_4\}\rightarrow \C$ the five paths that trace the values of $\sqrt[5]{f(a_0,\ldots,a_4)}$ is a closed loop.  
\end{claim}
\begin{proof}
Let $z=f(a_0,a_1,a_2,a_3,a_4)$ be an analytic function in complex variables $a_0,\dots,a_4$, and suppose that
for each $j=0,\dots, 4$ we have two closed loops
\[
	\beta_j:[0,1]\to \mathbb{C},\qquad \gamma_j:[0,1]\to \C
\]
that start and end at some fixed $a_j$, and such that $f\circ \beta_.$ and $f\circ \gamma_.$ avoid $0$.
Perform the path 
\[
	[\beta, \gamma]=\beta\gamma\beta^{-1}\gamma^{-1}; \quad (\text{Commutator})
\] 
on $a_0,\dots,a_4$ and follow the $5$ paths that trace the values of
\[
	\sqrt[5]{f(a_0,\dots,a_4)}.
\]
These paths are closed loops because both $\beta$ and $\gamma$ define a cyclic permutation of the $5$ radicals which commute.  
\end{proof} 

We now consider the simplest case of \textit{Arnold's} argument. 
\begin{Theo}
For any Monic polynomial $p(z)\in {\rm Poly}_2(\C)$ the roots of $p(z)$ cannot be expressed in using an expression of the form: 
\begin{equation}\label{function}
z=g\left(\sqrt[n_1]{f_1(a_0,\dots,a_4)},\sqrt[n_2]{f_2(a_0,\dots,a_4)},\dots,\sqrt[n_k]{f_k(a_0,\dots,a_4)}\right)
\end{equation}	
for some analytic functions $f_1,\dots,f_k:\C^5\to \C$, $g:\C^k\to \C$, and $n,k\in \Z$.	 
\end{Theo}

\begin{proof}
Fix distinct roots $z_1,\dots, z_5 \in \C$ of the polynomial $p(z)$. Construct continuous paths such that 
\[
	\hat{\beta}:(z_1,z_2,z_3,z_4,z_5) \mapsto (z_2,z_3,z_1,z_4,z_5)
\]
and 
\[
	\hat{\gamma}:(z_1,z_2,z_3,z_4,z_5) \mapsto (z_1,z_2,z_4,z_5,z_3)
\]
and denote by $\beta$, $\gamma$ the corresponding paths of the coefficients $a_0,\dots,a_4$ of $p(z)$.
Then as $\hat{\beta}$ permutes the roots then $\beta_j$ follows a closed loop by \textit{Vieta's} formulae. Similarly each $f_1(a_0,\ldots, a_4)\ldots f_k(a_0,\ldots, a_4)$ follow a closed loop under this motion and the paths $\sqrt[n_i]{f_i}$ amount to a cyclic permutation. As the argument is completely symmetrical we conclude the same for $\hat{\gamma}$. Hence, following $\hat{\beta}\hat\gamma\hat{\beta}^{-1}\hat{\gamma}^{-1}$, each of the paths of $\sqrt[n_i]{f_i}$ are closed. Therefore we have that $g(\sqrt[n_1]{f_1},\ldots \sqrt[n_i]{f_i})$ also follow a closed loop, but the roots $z_1,\ldots ,z_5$ realized a non-trivial permutation which means that $g(\sqrt[n_1]{f_1},\ldots \sqrt[n_i]{f_i})$ does not trace the roots $z_1,\ldots, z_5$.  
\end{proof}

\subsection{Necessity of Nested Radicals to Solve Cubic Equations} \label{sec:cubic}
\
\\
We now address the fact that a closed formula for the roots of a third degree polynomial requires the use of one level of nested radicals. On this screen the user can visualize the need for nested radicals in order to solve third degree polynomial equations\footnote{Screen shot of the JavaScript simulation for the case of a cubic polynomial (\figref{cubic}).}. We again construct closed paths on the set of roots of an arbitrary third degree polynomial equation, and apply the commutator of these paths so that we induce a non-trivial permutation on the set of roots, but both the coefficients and the values of these under the hypothetical formula that does not involve more than one level of nested radicals return to their original positions.   
\begin{figure}[H]
	\centering
	\includegraphics[scale=0.5]{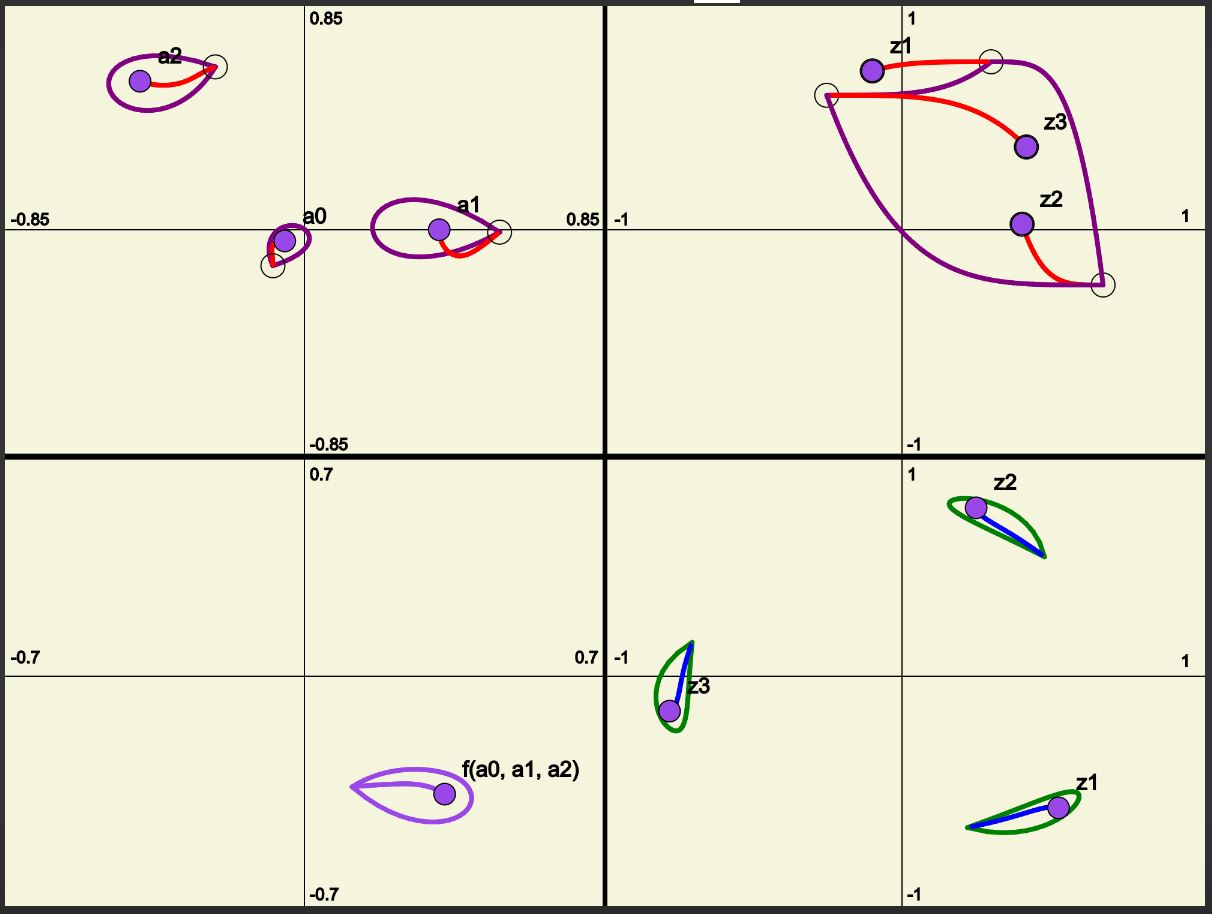}
	\caption[Cubic]{Case for Monic $p(z)$ in ${\rm Poly}_3(\mathbb{C})$}
	\label{fig:cubic}
	\end{figure} 

\begin{claim}
Any formula for the roots of a Monic polynomial $p(z) \in {\rm Poly}_3(\C)$ cannot be expressed by a formula $f:\{a_0,a_1,a_2\}\rightarrow \C^3$ for $f$ an analytic function and no nested radicals.
\end{claim}
\begin{proof}
Suppose otherwise, and consider a general Monic polynomial $p(z)\in {\rm Poly}_3(\mathbb{C})$ with roots $z_1,z_2,z_3$. As in (\ref{function}), no such formula can express the roots of $p(z)$. We can further construct paths that cyclically permute the roots $z_1,z_2,z_3$ while inducing closed loops for each coefficient of $p(z)$ and for the value of $f$. The paths of $\sqrt[n]{f}$ will either follow a closed loop or rotate by some angle $\theta$. Applying the commutator the values of $\sqrt[n]{f}$ undergo a rotation $\Delta \theta=0$, while inducing a non-trivial permutation of the roots $z_1,z_2,z_3$, contradiction. Therefore any formula requires nested radicals, given by \textit{Cardano's} formula.
\end{proof}
\subsection{Commutators in $S_n$}\label{sec:symmetricgroup}
\
\\
Our ability to solve algebraic equations using radicals is dependent on the solubility of special classes of groups. 
$S_n$ is the group of all permutations. $A_n$ is the group of all even permutations, where $A_n$ forms a subgroup of $S_n$. The reader can verify the following properties of the groups $S_n$ and $A_n$: 
\begin{enumerate}
	\item If $H$ is a subgroup of a soluble group $G$, then $H$ is soluble.
	\item If a group $G$ is not commutative and the only subgroups are unit element and $G$ itself, then $G$ is not soluble.
	\item For $n\geq 5$, $S_n$ contains a subgroup isomorphic to $A_5$.  
	\item Let $G$ be a finite group. Then $G$ is soluble if and only if there exists $n\in \mathbb{Z}$ such that $G^{(n)}=\{1\}$ (The $nth$ commutator group).
\end{enumerate}
The above properties allow us to conclude that $S_2$ is a commutative group and therefore soluble; $S_3$ is also a soluble group as $[[a,b],[c,d]]=\{e\}$ where $e$ is the identity element; $S_4$ is also a soluble group as $[[[a,b],[c,d]],[[a,b],[c,d]]]=\{e\}$, and lastly by properties $(1)$ and $(3)$ it follows that for $n\geq 5$, $S_n$ is not a soluble group.
\subsection{Quartic Equations}\label{sec:quartic}
\
\\
As in the case for third degree polynomial equations, there does exists a long and complicated formula that expresses the roots of an arbitrary degree four polynomial equation in terms of nested radicals, field operations, analytic functions, and the coefficients of the polynomial. 
\begin{figure}[htp]
	\centering
	\includegraphics[scale=0.5]{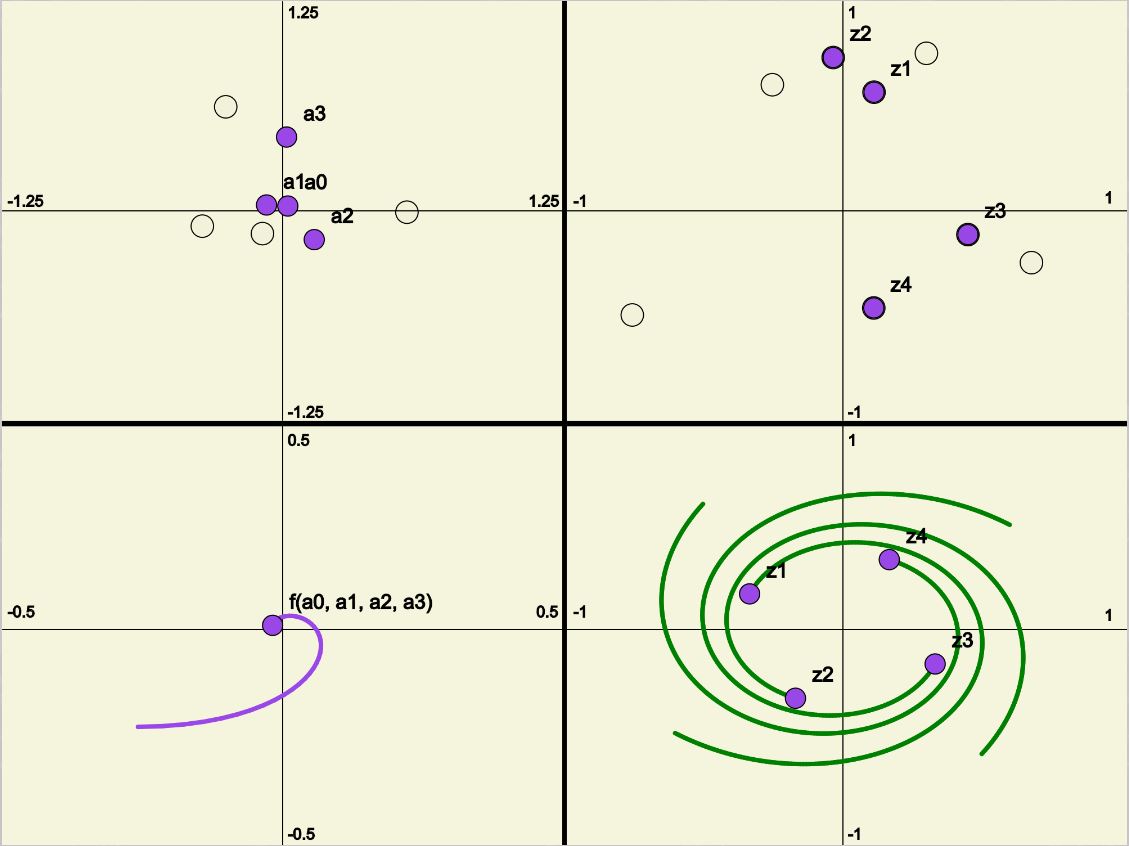}
	\caption[Quartic Case]{Case for Monic $p(z)$ in ${\rm Poly}_4(\C)$}
	\label{fig:quartic}
	\end{figure}
We can visually demonstrate that for a general fourth degree polynomial equation, by the fundamental theorem of algebra, there are in general four solutions $z_1,z_2,z_3,z_4$. Defining the set $A=\{z_1,z_2,z_3,z_4\}$ to be the set of roots of a general fourth degree equation, then what the visualization demonstrates is that we can rule out hypothetical formulas through the following process \footnote{Screen shot of the JavaScript simulation for the case of a fourth degree polynomial (\figref{quartic})}: 

There are 24 permutations of the set $A$ from which there are 12 non-trivial permutations of commutators of these 24 permutations. Taking a commutator of these 12 permutations we find that there are 4 non-trivial permutations of commutators of commutators, and finally taking commutators of commutators of commutators we arrive at a trivial permutation. Thus, this shows that for degree four polynomial equations any formula that will express the roots of this polynomial will require three levels of nested roots.
\subsection{Arnold's Theorem and the Impossibility of the Quintic in Radicals}\label{sec:main}
\begin{Theo}[Arnold's Theorem]
The Monodromy of the algebraic function $x(a)$ defined by the quintic equation $x^5+ax+1=0$ is the non-soluble group of the 120 permutations of five roots. That is, no function having the same topological branching type as $x(a)$ is re presentable as a finite combination of elementary functions and radicals \cite{MR2110624}. 
\end{Theo}
\subsection{Impossibility of the Quintic in Radicals}\label{sec:impossibility}
\
\\
In the final stage of our visualization we demonstrate the impossibility of constructing a formula that expresses the roots of a general degree five or higher polynomial equation in terms of analytic functions, finite field operations, and finite levels of nested radicals. By following the same processes as in previous cases, we construct closed paths such that their commutator induces non-trivial permutations of the set of roots, while the coefficients of the polynomial and their image under any analytic function with finite levels of nested radicals follow a closed loop \footnote{Screen shot of the JavaScript simulation for the case of a fifth degree polynomial (\figref{quintic})}.   
\begin{claim}
There does not exist a formula for the roots of a degree five polynomial equation built out a finite combination of analytic functions, field operations, and any level of nested radicals. 
\end{claim}
\begin{figure}[htp]
	\centering
	\includegraphics[scale=0.5]{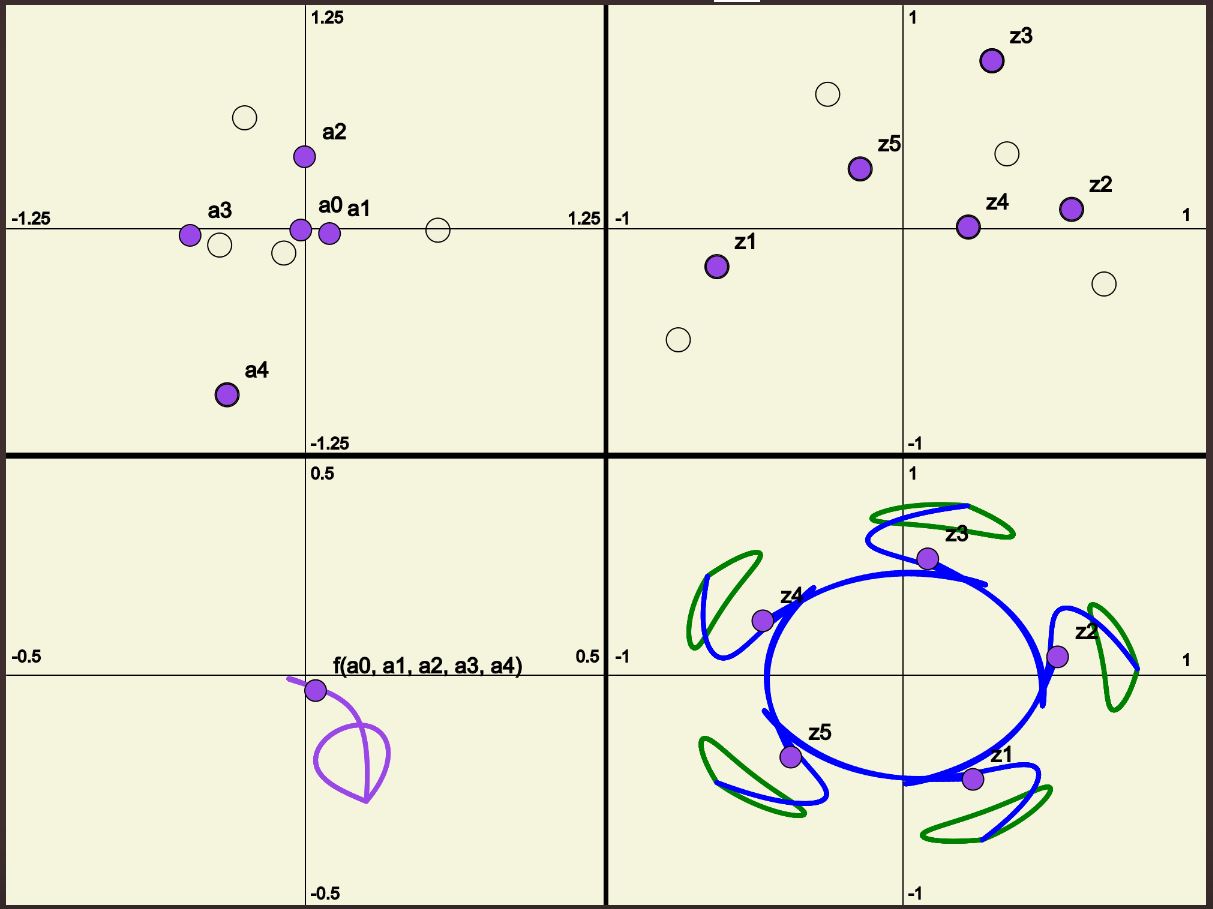}
	\caption{Case for Monic $p(z)$ in ${\rm Poly}_5(\mathbb{C})$}
	\label{fig:quintic}
	\end{figure} 
\begin{proof}
Suppose otherwise, and let $A=\{a_1,\ldots\,a_4\}$ be the set of coefficients of $p(z)\in$ ${\rm Poly}_5(\mathbb{C})$. Then, there exist functions $f_1,\ldots, f_k:A\rightarrow \mathbb{C}^5$ that expresses the roots of a general Monic polynomial $p(z)\in {\rm Poly}_5(\mathbb{C})$. By previous case there must be at least $N\geq 3$ levels of nested radicals for some $N\in \N$. Thus, suppose we have an expression of $N$ levels of nested radicals. Since $S_5$ is not soluble, then we know that there exists continuous paths such that their commutator induces a non-trivial permutation of the roots, while both the coefficients of $p(z)$ and $f_i$, for $1\leq i\leq k$, follow a closed loop, a contradiction. Therefore, there is no formula for the roots of a general polynomial equation constructed entirely out of a finite combination of analytic functions, field operations, and nested radicals. 
\end{proof}

\section{Implementation of Javascript}\label{sec:javascript}

We now discuss some of the details and technicalities of how respective javascript libraries are implemented and how the application was, in general, built.    
\begin{enumerate}
\item How the p5.js library operates and works to create the animation?

$\bullet$ There are two functions in the p5.js library which act as the driver program of the application. The setup function which is run once before the program begins and the draw function which runs continuously throughout the program executing the lines of code within. This repetition of the draw function is what allows for the creation of this animation as every time the function is run the canvas object is reset but the data attributes update their information simulating movement. This is similar to how camera reels work as every frame is a still image but when these frames are run in succession to one another at a certain frame rate the images appear to be moving.

\item How the classes work with each other?

There are three classes in the program which all build off one another to create a motion in the canvas. These are the Point class, the Path class, and the Motion class.

$\bullet$ The Point class is what allows to plot the points on the complex plane. When an instance of the Point class is instantiated it takes in a complex number written in terms of real and imaginary components as well as the quadrant in which the point is to be drawn (I-IV). The real and imaginary component translates into real and imaginary pixel coordinates through a method defined within the class, pixelToPoint, relative to the quadrant specified and the quadrants respective scale. The pixel locations map to the canvas object and where the points are drawn through another method plotPoint, which draws an ellipse of a predefined size at a given pixel coordinate.

$\bullet$ The Path class is what allows movement between two locations of a given point. The class itself takes as parameters two complex roots and a quadrant which are then used to create two instances of Point objects. One point is used as a starting location while the other point is used as an ending location. There is a method in Path named startPath which begins the movement of the starting point to the ending point, this is accomplished by incrementing the real and imaginary pixel component of the start point by a small size for each iteration of the draw function until it reaches the end point. There is also a method named setEnd which is what allows these paths to take on new end points which allow the points to travel to different locations after they have reach their destination.
 
$\bullet$ The Motion class is the set of movements of all paths in a quadrant. An instantiation of this class takes as parameters n complex roots that are then used to build n different instances of the Path class each with a starting and ending point. This class is meant to act as a controller for all the paths within a set. The Motion class calls on methods defined in the Path class to coordinate movement between the n different points. This class can update all existing paths with new endpoints, begin movement of the points, and be able to plot the points with methods defined as setNewEnds, update, and plot, respectively. 

\item How do we implement Vieta's formula? 

A function that takes a list of complex numbers and calculate each part of Vieta's formulas storing each coefficient calculated into a list that then returns at the end of the function. The coefficients returned are used to create instances of the Point class which are then plotted on the second quadrant of the application using the method, plotPoint. As the solutions in the first quadrant move, the resulting data attributes of their real and imaginary components are updated and sent to this function and the coefficients are renewed appropriately with each repetition of the draw function.
We use a library from GitHub named Complex.js \cite{infusion_2019}.\\ 

\item From the coefficients how are the solutions calculated?\\

A function exists to take as a parameter the list of coefficients that is calculated from Vieta's formula. This function takes in the coefficients of the polynomial (which are calculated from symmetric expressions via Vieta) and traces the values of the function as the coefficients change in a continuous way. A new instantiation of the Point class is calculated using the resulting real and imaginary component and is then plotted on quadrant III of the application. The point is updated along with the coefficients from Vieta's formula as the draw function repeats.

\item The loop in quadrant IV? 

A function exists that takes as a parameter a single complex point obtained from quadrant III, calculates the roots of the complex point and plots them on Quadrant IV. This plotting follows a closed loop so that at the end of the induced permutation of the roots, the rotations end up where it started, but the roots themselves have exchanged places. These roots are found first by calculating the magnitude and the angle it makes with the real axis. In general, a loop is then iterated n times where the formula for nth distinct roots of $z$ is given by (\ref{polar}).
The resulting root is then converted into an instance of the Point class and plotted on Quadrant IV and updated appropriately as the draw function repeats.

\item How are the traces drawn?  

As the program runs, the real and imaginary components of instances of the Point class are continuously updated and likewise are their pixel counterparts. The pixel data attributes are stored every three frames into a list that grows as the program is run. A function is run on this list, named drawHistory, which iterates through each pixel and draws a line from one pixel location to its successor stopping before the final point.
\end{enumerate}
Although, the purpose of this application was to illustrate an abstract mathematical argument in a geometric and intuitive manner, we did, however, project everything onto the complex plane. Future work includes working in higher dimensions (three dimensional space or four dimensional space with coloring) as well as building a mathematical program that graphs the Riemann surface of the radical function; in order to illustrate in more generality the argument given by \textit{Arnold}. 
\section*{Acknowledgments}

We would like express our gratitude to Professor Alexander Furman for his immense commitment, dedication, and guidance in the creation of our project; to Professor David Dumas for giving us the opportunity to participate in the \textit{Mathematical Computing Laboratory} (MCL)\cite{mathematicalcomputinglaboratory}, his support, and encouragement throughout the creation of this project. We thank the \textit{Department of Mathematics, Statistics, and Computer Science} at the \textit{University of Illinois at Chicago} for their support. A special thanks to Professor Ramin Takloo-Bighash for his dedication towards the advancement of undergraduate students in the field of mathematics as well as for his guidance in the process of writing and submission of this paper.    

\bibliographystyle{plain}  
\bibliography{abelreferences}

\begin{thebibliography}{1}

\bibitem{mathematicalcomputinglaboratory}
Mathematical computing laboratory at uic.
\newblock \url{https://mcl.math.uic.edu/}.

\bibitem{MR2110624}
V.~B. Alekseev.
\newblock {\em Abel's theorem in problems and solutions}.
\newblock Kluwer Academic Publishers, Dordrecht, 2004.
\newblock Based on the lectures of Professor V. I. Arnold, With a preface and
  an appendix by Arnold and an appendix by A. Khovanskii.

\bibitem{alekseev2004abel}
Valerij~Borisovi{\v{c}} Alekseev.
\newblock {\em Abel's Theorem in Problems and Solutions: Based on the lectures
  of Professor VI Arnold}.
\newblock Springer Science \& Business Media, 2004.

\bibitem{infusion_2019}
Infusion.
\newblock infusion/complex.js, Apr 2019.
\newblock \url{https://github.com/infusion/Complex.js/}.

\bibitem{ostrander_morales_kalicki_2019}
Ryan Ostrander, Juan Morales, and Veronica Kalicki.
\newblock ryanostrander/abel-s-theorem, Apr 2019.
\newblock \url{https://github.com/ryanOstrander/Abel-s-Theorem}.

\end{thebibliography}

\end{document}